\documentclass[10pt]{article}
\usepackage{amssymb}
\usepackage{graphicx}
\usepackage{xcolor} 
\usepackage{tensor}
\usepackage{fullpage} 
\usepackage{amsmath}
\usepackage{amsthm}
\usepackage{verbatim}
\usepackage{txfonts,bm,mathtools}

\usepackage{enumitem}
\setlist[enumerate]{leftmargin=1.5em}
\setlist[itemize]{leftmargin=1.5em}

\usepackage{hyperref}
\usepackage{seqsplit,mathtools} 
\usepackage{caption}
\usepackage{subcaption}

\mathtoolsset{showonlyrefs}

\definecolor{green}{rgb}{0,0.8,0} 

\newcommand{\Blue}[1]{\begingroup\color{blue} #1\endgroup} 

\newtheorem{thm}{Theorem}[section]

\newtheorem{lem}[thm]{Lemma}
\newtheorem{prop}[thm]{Proposition}

\theoremstyle{definition}

\theoremstyle{remark}
\newtheorem{rmk}[thm]{Remark}

\numberwithin{equation}{section}
\newcommand{\nrm}[1]{\Vert#1\Vert}

\newcommand{\tld}[1]{\widetilde{#1}}
\newcommand{\br}[1]{\overline{#1}}

\newcommand{\nnrm}[1]{{\vert\kern-0.25ex\vert\kern-0.25ex\vert #1 
		\vert\kern-0.25ex\vert\kern-0.25ex\vert}}

\newcommand{\supp}{{\mathrm{supp}}\,}

\newcommand{\rd}{\partial}
\newcommand{\nb}{\nabla}

\newcommand{\ift}{\infty}

\newcommand{\alp}{\alpha}
\newcommand{\bt}{\beta}

\newcommand{\dlt}{\delta}

\newcommand{\tht}{\theta}

\newcommand{\omg}{\omega}
\newcommand{\Omg}{\Omega}

\newcommand{\bbN}{\mathbb N}

\newcommand{\bbR}{\mathbb R}

\newcommand{\calA}{\mathcal A}

\newcommand{\calF}{\mathcal F}

\newcommand{\calK}{\mathcal K}

\newcommand{\To}{\longrightarrow}

\begin{document}

\bibliographystyle{plain}
\title{Global regularity of some axisymmetric, single-signed 
	vorticity\\in any dimension}
\author{
	Deokwoo Lim\thanks{The Research Institute of Basic Sciences, Seoul National University, 1 Gwanak-ro, Gwanak-gu, Seoul 08826, Republic of Korea. Email: dwlim95@snu.ac.kr
} }

\date\today
\maketitle

\begin{abstract}
	We consider incompressible Euler equations in any dimension $ d\geq3 $ imposing axisymmetric symmetry without swirl. While the global regularity of smooth flows in this setting has been well-known in $ d=3 $, the same question in higher dimensions $ d\geq4 $ remains unsolved. 
	Recently, global regularity for the case $ d=4 $ with some extra decay assumption on vorticity is obtained by proving global estimate of the radial velocity. Now we prove that the vorticity with single-sign and a similar decay assumption is globally regular for any $ d\geq4 $. This is due to pointwise decay estimate of radial velocity in sufficiently large radial distance, which depends on time. The result is of confinement type for support growth, which is going back to Marchioro \cite{M94} and Iftimie--Sideris--Gamblin \cite{ISG99} for $ \bbR^{2} $. In particular, we follow the approach of Maffei--Marchioro \cite{Maffei2001} for $ d=3 $ so that we generalize the confinement into any dimension.
\end{abstract}

\renewcommand{\thefootnote}{\fnsymbol{footnote}}
\renewcommand{\thefootnote}{\arabic{footnote}}

\section{Introduction}\label{sec_intro}

\subsection{Axisymmetric, swirl-free Euler equations}\label{subsec_ase}

We consider the incompressible Euler equations in $ \bbR^{d} $ for dimensions $ d\geq3 $:
\begin{equation}\label{eq_Euler}
	\left\{
	\begin{aligned}
		\rd_{t}u+(u\cdot\nb)u+\nb p&=0,\\
		\nb\cdot u&=0.
	\end{aligned}
	\right.
\end{equation}
Here, our concern is \textit{axisymmetric} and \textit{swirl-free} flows, where the velocity field $ u $ takes the form
\begin{equation}\label{eq_axisymformofu}
	u=u^{r}(r,z)e^{r}+u^{z}(r,z)e^{z},
\end{equation}
in the cylindrical coordinate system $ (r,\tht_{1},\tht_{2},\cdots,\tht_{d-3},\phi,z) $ in $ \bbR^{d} $. Now we denote the scalar vorticity $ \omg $, which is defined as $ \omg:=\rd_{z}u^{r}-\rd_{r}u^{z} $. Then the Euler equations \eqref{eq_Euler} becomes the following vorticity equation:

\begin{align}
	\rd_{t}\omg+u\cdot\nb\omg&=\frac{(d-2)u^{r}}{r}\omg,\label{eq_vortformNd}\\
\end{align}
with
\begin{equation}\label{eq_urpsi}
	\begin{split}
		u^{r}&=-\frac{1}{r^{d-2}}\rd_{z}\psi, \qquad u^{z}=\frac{1}{r^{d-2}}\rd_{r}\psi,
	\end{split}
\end{equation}
where $ \psi $ is the scalar stream function that is given in the next section. 
In addition, the divergence-free condition in $ \bbR^{d} $ can be written as
$$ \nb\cdot u=\rd_{r}u^{r}+\rd_{z}u^{z}+\frac{d-2}{r}u^{r}=0. $$


\noindent Interestingly, if we define the \textit{relative} vorticity as $ \xi:=r^{-(d-2)}\omg $, the vorticity equation \eqref{eq_vortformNd} and the divergence-free condition above lead $ \xi $ to satisfy the following simple transport equation
\begin{equation}\label{eq_transport}
	\rd_{t}\xi+u\cdot\nb\xi=0,
\end{equation}
which makes the quantity $ \nrm{\xi}_{L^{p}(\bbR^{d})} $ to be conserved in time for any $ p\in[1,\ift] $.








\subsection{Main result}\label{subsec_result}

%
%
In $ d=3 $, the global regularity of axisymmetric and swirl-free smooth solutions of the Euler equations \eqref{eq_Euler} was established by Ukhovskii--Yudovich \cite{UY1968} (also see Danchin \cite{Danaxi}). 
%
%
%
%
The following theorem presents that the single-sign and the compact support of $ \omg $ implies global regularity of $ \omg $ for any $ d\geq3 $. 
We assume $ u_{0}\in H^{s}(\bbR^{d}) $ for some $ s>1+d/2 $ to use the local well-posedness (\cite{Ka,KL,Miller}) of the corresponding solution $ u $ of \eqref{eq_Euler}.


\begin{thm}\label{thm:3}
	Let $ d\geq3 $, and assume that $ \omg_{0} $ is single-signed, compactly supported in $ \bbR^{d} $, and satisfies \linebreak $ r^{-(d-2)}\omg_{0}\in L^{\ift}(\bbR^{d}) $. Then 
	the corresponding solution $ \omg $ of \eqref{eq_vortformNd} is global in time. 
	In particular, for any $ t\geq0 $, it satisfies
	\begin{equation}\label{eq_omgliftestRd}
		\nrm{\omg(t)}_{L^{\ift}(\bbR^{d})}\leq C[(1+t)\ln(e+t)]^{(d-2)/(d+1)},
	\end{equation}
	for some $ C>0 $ depending only on $ d $, $ \nrm{r\omg_{0}}_{L^{1}(\bbR^{d})} $, $ \nrm{r^{-(d-2)}\omg_{0}}_{L^{\ift}(\bbR^{d})} $, and $ S_{0}:=\sup\lbrace r>0 : (r,z)\in\supp\omg_{0}\rbrace $.
\end{thm}
\begin{rmk}\label{rmk_actualest}
	Theorem \ref{thm:3} is a direct consequence of the following estimate of the support size in the radial direction
	\begin{equation}\label{eq_defofSt}
		S(t):=\sup_{s \in [0,t]} \sup\lbrace r>0 : (r,z)\in\supp \omg(s)\rbrace,
	\end{equation}
	with $ S(0)=S_{0} $:
	\begin{equation}\label{eq_StestinRd}
		S(t)\leq C[(1+t)\ln(e+t)]^{1/(d+1)}\quad\text{for any}\quad t\geq0,
	\end{equation}
	for some $ C>0 $, which depends only on the above four quantities. Indeed, using the conservation of $ \nrm{\xi(t)}_{L^{\ift}(\bbR^{d})}=\nrm{r^{-(d-2)}\omg(t)}_{L^{\ift}(\bbR^{d})} $ in time and \eqref{eq_StestinRd}, we obtain \eqref{eq_omgliftestRd}:
	$$ \nrm{\omg(t)}_{L^\infty(\bbR^d)} \leq \bigg\|\frac{\omg(t)}{r^{d-2}}\bigg\|_{L^\infty(\bbR^d)}  \cdot S(t)^{d-2}= \bigg\|\frac{\omg_{0}}{r^{d-2}}\bigg\|_{L^\infty(\bbR^d)}  \cdot S(t)^{d-2}\leq C[(1+t)\ln(e+t)]^{(d-2)/(d+1)}. $$
\end{rmk}

\subsection{Literature review and key ideas}\label{subsec_litkey}

\noindent \textbf{BKM criterion in general dimensions and axisymmetric, swirl-free solutions.} The local well-posedness of a solution $ u\in C\big([0,T_{\text{max}});H^{s}(\bbR^{d})\big)\cap C^{1}\big([0,T_{\text{max}});H^{s-1}(\bbR^{d})\big) $ of \eqref{eq_Euler} for $ s>1+d/2 $ is established in Kato \cite{Ka}, Kato--Lai \cite{KL} (also see Miller \cite{Miller}), where $ T_{\text{max}}\in(0,\ift] $ is the maximal time of existence of $ u $. Naturally, one can ask whether $ T_{\text{max}}<\ift $ could hold in certain cases or not. The Beale--Kato--Majda criterion provides an answer for that question (\cite{BKM} for $ d=3 $, \cite{KaPo} for any $ d\geq3 $). In 3D, it says that $ T_{\text{max}}=\ift $ if for any $ T\in(0,\ift) $, there exists $ M\in(0,\ift) $ such that the vorticity field $ \nb\times u $ satisfies
$$ \int_{0}^{T}\nrm{\nb\times u(t)}_{L^{\ift}(\bbR^{3})}dt\leq M. $$
In higher dimensions, the integrand is replaced by the $ L^{\ift}(\bbR^{d}) $-norm of the rotation matrix $ \Omg:=\frac{1}{2}[(\nb u)-(\nb u)^{\text{tr}}] $.  
On the other hand, when we consider axisymmetric and swirl-free solutions, 
the rotation matrix is expressed as
\begin{equation*}
	\Omg=\frac{1}{2}\omg(e^{r}\otimes e^{z}-e^{z}\otimes e^{r}),
\end{equation*}
where $ e^{r}\otimes e^{z} $ is the tensor product between $ e^{r} $ and $ e^{z} $, so the term $ \nrm{\Omg}_{L^{\ift}(\bbR^{d})} $ becomes equivalent with the norm $ \nrm{\omg}_{L^{\ift}(\bbR^{d})} $ of the scalar vorticity $ \omg $. That is, showing that $ \nrm{\omg(t)}_{L^{\ift}(\bbR^{d})} $ is bounded by some regular function of $ t $ is sufficient enough to prove the global regularity of $ \omg $.

\medskip

\noindent The big difference between the 3D case and higher dimensional case is that in $ d\geq4 $, the relative vorticity $ \xi=r^{-(d-2)}\omg $ is, in general, not in $ L^{\ift}(\bbR^{d}) $ even when $ \omg $ is smooth. This is one of the reasons why the global regularity of axisymmetric, swirl-free smooth solutions of \eqref{eq_Euler} in any $ d\geq4 $ is still unknown.

\medskip



\noindent\textbf{Support growth of a single-signed vorticity in 2D.} 
When the vorticity is allowed to be single-signed, the confinement argument of Marchioro \cite{M94} says that one can improve the support growth estimate from a \textit{pointwise} decaying estimate of the radial component $ \frac{x}{|x|}\cdot u $ of the velocity field. For a vortex patch in $ \bbR^{2} $, as in \cite{M94}, one may simply prove that the support is contained in the region $ \lbrace |x|\leq C(1+t)^{1/2}\rbrace $. 
There are two key ideas in showing this estimate. First, the angular impulse
$$ \int_{\bbR^{2}}|x|^{2}\omg(x)dx,
$$
which is conserved in time, is equal to the norm $ \nrm{|x|^{2}\omg}_{L^{1}(\bbR^{2})} $, when we have only one sign. Next, the conservation implies the following time-independent estimate in the local region:
$$ \int_{||x|-R|\leq R/2}\omg(t,x)dx\leq\int_{|x|> R/2}\omg(t,x)dx\lesssim\frac{1}{R^{2}}\int_{\bbR^{2}}|x|^{2}\omg_{0}(x)dx. $$
This enables us to obtain a decay rate of $ \frac{x}{|x|}\cdot u $. The simple estimate with exponent $ 1/2 $ is improved in \cite{M94} with exponent $ 1/3 $ and by Iftimie--Sideris--Gamblin \cite{ISG99} with exponent $ 1/4 $, involving a logarithmic term. 
Two important ideas were used to obtain these refinements. 
One is the dynamics of $ \omg $, that is, $ \omg $ is a solution of the 2D Euler equations. This is used to calculate the time-derivative of the term
$$ \frac{d}{dt}\int_{\bbR^{2}}\phi_{R}(y)\omg(t,y)dy=\int_{\bbR^{2}}\nb\phi_{R}(y)\cdot u(t,y)\omg(t,y)dy, $$
where $ \phi_{R} $ is a suitable smooth, non-negative, bounded, radial, and increasing function that satisfies $ 1_{\lbrace|y|>R\rbrace}\leq\phi_{R} $. The other is the symmetrization of the term in the right-hand side, which is a double-integral:
$$ \int_{\bbR^{2}}\nb\phi_{R}(y)\cdot u(t,y)\omg(t,y)dy\simeq\int_{\bbR^{2}}\int_{\bbR^{2}}\calK_{R}(y,\br{y})\omg(t,\br{y})\omg(t,y)d\br{y}dy, $$
for some kernel $ \calK_{R} $. The function $ \phi_{R} $ is an approximation of the simple function $ 1_{\lbrace|x|>R\rbrace} $ in some sense (such as $ L^{1} $), so roughly speaking, what we have is an estimate of the term
$$ \frac{d}{dt}\int_{|x|>R}\omg(t,x)dx. $$
Then we make $ R $ to depend on $ t $. We refer the reader to Serfati \cite{Serfati_pre} in $ \bbR^{2} $, and Iftimie \cite{Iftimie1999}, Iftimie--Lopes Filho--Nussenzveig Lopes \cite{ILL2003,ILL2007}, and Choi--Denisov \cite{CD2019} in other 2D domains as well.
\medskip

\noindent\textbf{Support growth of a single-signed vorticity in higher dimensions.} 
In the 3D axisymmetric, swirl-free setting with single-signed, compactly supported vorticity in $ L^{\ift}(\bbR^{3}) $, Maffei--Marchioro \cite{Maffei2001} followed the spirit of the above argument to show that the vorticity is contained in the region $ \lbrace r\leq C(1+t)^{1/4}\ln (e+t)\rbrace $. This is because the fluid impulse
$$ \int_{\bbR^{3}}x\times\omg(x)dx $$
is conserved in time, and in the axisymmetric, swirl-free setting, this yields the time-preservation of the radial impulse
$$ \int_{\bbR^{3}}r\omg(x)dx\simeq \int_{0}^{\ift}\int_{-\ift}^{\ift}r^{2}\omg(r,z)dzdr. $$
Again, by the sign condition of $ \omg $, this quantity is the same as the term $ \nrm{r\omg}_{L^{1}(\bbR^{3})} $. 
\medskip

\noindent We slightly improve the logarithmic power of \cite{Maffei2001} for $ d=3 $ and generalize this argument up to any 
higher dimensional case $ d\geq3 $ to get the 
bound \eqref{eq_StestinRd}. 
Once again, this is thanks to the conservation of the radial impulse:
$$ \int_{\bbR^{d}}r\omg(x) dx\simeq_{d}\int_{0}^{\ift}\int_{-\ift}^{\ift}r^{d-1}\omg(r,z)dzdr. $$
While this quantity is conserved in general, it is equivalent with the norm $ \nrm{r\omg}_{L^{1}(\bbR^{d})} $ only when $ \omg $ is single-signed.
\medskip

\noindent \textbf{Related works in higher dimensions.} In the previous work of the author in collaboration with Choi and Jeong \cite{CJLglobal22}, 
we showed the global regularity in $ d=4 $ of axisymmetric, swirl-free solution $ \omg $, which shows decay as $ r\to\ift $ and goes to $ 0 $ sufficiently fast as $ r\to0 $. In addition, we proved that if $ \omg $ is single-signed and compactly supported, then it is globally regular, at least up to $ d\leq 7 $. Such restriction in dimension was made because we used the Feng--Sverak (\cite{FeSv}) type global estimate of $ u^{r} $
$$ \nrm{u^{r}}_{L^{\ift}(\bbR^{d})}\lesssim_{d}\nrm{r^{d-2}\omg}_{L^{1}(\bbR^{d})}^{1/4}\bigg\|\frac{\omg}{r^{d-2}}\bigg\|_{L^{1}(\bbR^{d})}^{1/4}\bigg\|\frac{\omg}{r^{d-2}}\bigg\|_{L^{\ift}(\bbR^{d})}^{1/2}, $$
which gave us for any $ t\geq0 $,
$$ S'(t)\lesssim_{d} S(t)^{(d-3)/4}\nrm{r\omg_{0}}_{L^{1}(\bbR^{d})}^{1/4}
$$
In the next section, 
we overcome this restriction $ d\leq7 $ in dimension by using the confinement argument.

\medskip

\noindent 
Recently, Gustafson--Miller--Tsai \cite{GMT2023} obtained an upper and lower bound of the radial impulse 
for global time in $ d=3, 4 $, and for local time in any $ d\geq5 $. In particular, their lower bound improves the vorticity growth result by Choi--Jeong \cite{CJ-axi} in $ d=3 $. As a note, for the case $ d=3 $, Childress--Gilbert--Valiant \cite{ChilGil1} and Childress--Gilbert \cite{ChilGil2} suggest that $ \nrm{\omg(t)}_{L^{\ift}(\bbR^{3})} $ and its related quantities can blow up in infinite time.

\medskip

\noindent In section \ref{sec_supportpatch}, we shall prove a time-independent pointwise estimate of $ u^{r} $ when its associated relative vorticity is a compactly supported patch as a warm-up calculation. We included this estimate to convince the readers that global regularity is very natural for one-signed vorticity. In section \ref{sec_singlesign}, using the dynamics of the Euler flow, we will show several technical lemmas that are required for proving a refined pointwise estimate of $ u^{r} $, which holds for sufficiently large $ r=r(t)>0 $. 
In section \ref{sec_pfmain}, the last section, we finish the proof of this estimate and show Theorem \ref{thm:3}.

\section{Support growth of a patch for any $ d\geq3 $}\label{sec_supportpatch}

\subsection{The explicit form of $ u^{r} $}\label{subsec_patchexp}

We let $ d\geq3 $ and denote the $ (r,z) $-half plane as
$$\Pi := \left\{ (r,z) : r\ge0, z \in \bbR \right\}.$$
First, let us figure out the explicit form of $ u^{r} $. To do this, we need the exact form of the scalar stream function $ \psi $, where its precise derivation can be found in \cite[Sec. 3]{Miller}. 
The stream function $ \psi $ is given as
$$ \psi(r,z) = \iint_{\Pi} G_{d}(r,z,\br{r},\br{z}) \omg(\br{r},\br{z}) d\bar{z} d\bar{r}, $$
with the kernel
\begin{equation*}
	\begin{split}
		G_{d}(r,z,\br{r},\br{z})&:=c_{d}\int_{0}^{\pi}\frac{(r\br{r})^{d-2}\cos\tht\sin^{d-3}\tht}{[r^{2}+\br{r}^{2}-2r\br{r}\cos\tht+(z-\br{z})^{2}]^{d/2-1}}d\tht\\
		&=c_{d}(r\br{r})^{d/2-1}F_{d}\bigg(\frac{(r-\br{r})^{2}+(z-\br{z})^{2}}{r\br{r}}\bigg),
	\end{split}
\end{equation*}
for some constant $ c_{d}>0 $ and
\begin{equation}\label{eq_Fd}
	F_{d}(s):=\int_{0}^{\pi}\frac{\cos\tht\sin^{d-3}\tht}{[2(1-\cos\tht)+s]^{d/2-1}}d\tht,\quad s>0.
\end{equation}
Then from \eqref{eq_urpsi}, we can obtain the explicit form of $ u^{r} $:
\begin{equation}\label{eq_urform}
	\begin{split}
		u^{r}(r,z)&=-\frac{1}{r^{d-2}}\iint_{\Pi}\rd_{z}G_{d}(r,z,\br{r},\br{z})\omg(\br{r},\br{z})d\br{z}d\br{r} \simeq_{d}\iint_{\Pi}\frac{\br{r}^{d/2-2}(z-\br{z})}{r^{d/2}}F_{d}'\bigg(\frac{(r-\br{r})^{2}+(z-\br{z})^{2}}{r\br{r}}\bigg)\omg(\br{r},\br{z})d\br{z}d\br{r}.
	\end{split}
\end{equation}

\noindent The following is a key estimate of $ F_{d}' $ from \cite{CJLglobal22}. This is a generalization of the estimate in $ d=3 $ from \cite{FeSv}. Also see the recent paper \cite{GMT2023} as well.
\begin{lem}[{{\cite[Lem. 2.2]{CJLglobal22}}}]\label{lem_Fd'est}
	$ F_{d}' $ satisfies
	\begin{equation}\label{eq_Fdprime}
		|F_{d}'(s)|\lesssim_{d}\min\bigg\lbrace\frac{1}{s}, \frac{1}{s^{d/2+1}}\bigg\rbrace,\quad s>0.
	\end{equation}
\end{lem}
\noindent This was obtained by expanding $ F_{d}' $ near $ s=0 $ and $ s=\ift $. 
We use this estimate of $ F_{d}' $ to obtain a \textit{pointwise} estimate of $ u^{r} $ and show Theorem \ref{thm:3}.

\subsection{A time-independent pointwise estimate of $ u^{r} $ for a patch}\label{subsec_patchtimeind}

If we are interested only in global regularity with a rough estimate, then adapting the early work of \cite{M94} for 2D vortex patch, we can generate a time-independent pointwise estimate of $ u^{r} $, 
which holds for \textit{any} $ (r,z)\in\Pi $ when the relative vorticity $ \xi=r^{-(d-2)}\omg $ is given as a patch $ 1_{\calA}, \ \calA\subset \Pi $.

\begin{prop}\label{prop_simpleur}
	Let $ d\geq3 $ and let $ \xi=1_{\calA} $ for some compact set $ \calA\subset\Pi
	$. 
	Then there exists $ C>0 $ such that for any $ (r,z)\in\Pi $, we have
	\begin{equation}\label{eq_simpleur}
		|u^{r}(r,z)|\leq C\bigg[\frac{1}{r^{d}}
		+\frac{1}{\sqrt{r}}\bigg],
	\end{equation}
	where $ C>0 $ depends only on $ d $, $ \nrm{r^{d-2}\xi}_{L^{1}(\Pi)} $, and $ \nrm{r^{2d-3}\xi}_{L^{1}(\Pi)} $. 
\end{prop}

\begin{rmk}
	For a solution $ \xi(t) $ of \eqref{eq_transport} with initial data $ \xi_{0} $, assuming its existence, the quantities
	$$ \nrm{r^{d-2}\xi}_{L^{1}(\Pi)}\simeq_{d}\nrm{\xi}_{L^{1}(\bbR^{d})},\quad\nrm{r^{2d-3}\xi}_{L^{1}(\Pi)}\simeq_{d}\nrm{r\omg}_{L^{1}(\bbR^{d})} $$
	are conserved in time.
\end{rmk}

\begin{proof}[Proof of Proposition \ref{prop_simpleur}]
	We fix $ (r,z)\in\Pi $. 
	Let us recall the form \eqref{eq_urform} of $ u^{r} $, which in our case becomes
	\begin{equation*}
		\begin{split}
			u^{r}(r,z)&=\iint_{\Pi}\frac{\br{r}^{d/2-2}(z-\br{z})}{r^{d/2}}F_{d}'\bigg(\frac{(r-\br{r})^{2}+(z-\br{z})^{2}}{r\br{r}}\bigg)\omg(\br{r},\br{z})d\br{z}d\br{r}\\
			&=\iint_{\Pi}\frac{\br{r}^{d/2-2}(z-\br{z})}{r^{d/2}}F_{d}'\bigg(\frac{(r-\br{r})^{2}+(z-\br{z})^{2}}{r\br{r}}\bigg)\br{r}^{d-2}\xi(\br{r},\br{z})d\br{z}d\br{r}\\
			&=\iint_{\calA}\frac{\br{r}^{(3d)/2-4}(z-\br{z})}{r^{d/2}}F_{d}'\bigg(\frac{(r-\br{r})^{2}+(z-\br{z})^{2}}{r\br{r}}\bigg)d\br{z}d\br{r}.
		\end{split}
	\end{equation*}
	We split the integral domain $ \calA $ into
	\begin{equation*}
		\begin{split}
			A_{1}&:=\calA\cap\bigg\lbrace(\br{r},\br{z})\in\Pi : |r-\br{r}|>\frac{r}{2}\bigg\rbrace,\\
			A_{2}&:=\calA\cap\bigg\lbrace(\br{r},\br{z})\in\Pi : |r-\br{r}|\leq\frac{r}{2},\;|z-\br{z}|>\frac{r}{2}\bigg\rbrace,\\
			A_{3}&:=\calA\cap\bigg\lbrace(\br{r},\br{z})\in\Pi : |r-\br{r}|\leq\frac{r}{2},\;|z-\br{z}|\leq \frac{r}{2}\bigg\rbrace.
		\end{split}
	\end{equation*}
	Then in $ A_{1} $, we have
	\begin{equation*}
		\begin{split}
			\bigg|\iint_{A_{1}}\frac{\br{r}^{(3d)/2-4}(z-\br{z})}{r^{d/2}}F_{d}'\bigg(\frac{(r-\br{r})^{2}+(z-\br{z})^{2}}{r\br{r}}\bigg)d\br{z}d\br{r}\bigg|&\leq\iint_{A_{1}}\frac{\br{r}^{(3d)/2-4}|z-\br{z}|}{r^{d/2}}\cdot\frac{(r\br{r})^{d/2+1}}{[(r-\br{r})^{2}+(z-\br{z})^{2}]^{d/2+1}}d\br{z}d\br{r}\\
			&\leq\iint_{A_{1}}\frac{r\br{r}^{2d-3}}{[(r-\br{r})^{2}+(z-\br{z})^{2}]^{(d+1)/2}}d\br{z}d\br{r}\\
			&\lesssim_{d}
			\frac{1
				}{r^{d
				}}\cdot
			\nrm{r^{2d-3}\xi}_{L^{1}(\Pi)}.
		\end{split}
	\end{equation*}
	In $ A_{2} $, we get
	\begin{equation*}
		\begin{split}
			\bigg|\iint_{A_{2}}\frac{\br{r}^{(3d)/2-4}(z-\br{z})}{r^{d/2}}F_{d}'\bigg(\frac{(r-\br{r})^{2}+(z-\br{z})^{2}}{r\br{r}}\bigg)d\br{z}d\br{r}\bigg|&\leq\iint_{A_{2}}\frac{\br{r}^{(3d)/2-4}|z-\br{z}|}{r^{d/2}}\cdot\frac{r\br{r}}{(r-\br{r})^{2}+(z-\br{z})^{2}}d\br{z}d\br{r}\\
			&\leq\iint_{A_{2}}\frac{\br{r}^{(3d)/2-3}}{r^{d/2-1}}\cdot\frac{1}{[(r-\br{r})^{2}+(z-\br{z})^{2}]^{1/2}}d\br{z}d\br{r}\\
			&\lesssim
			\iint_{A_{2}}\underbrace{\bigg(\frac{\br{r}}{r}\bigg)^{d/2-1}}_{\lesssim_{d}1
			}\cdot\frac{\br{r}^{d-2}}{r}d\br{z}d\br{r}\lesssim_{d}
			\frac{1
			}{r}\cdot\nrm{r^{d-2}\xi}_{L^{1}(\Pi)}.
		\end{split}
	\end{equation*}
	In $ A_{3} $, we have
	\begin{equation*}
		\begin{split}
			\bigg|\iint_{A_{3}}\frac{\br{r}^{(3d)/2-4}(z-\br{z})}{r^{d/2}}F_{d}'\bigg(\frac{(r-\br{r})^{2}+(z-\br{z})^{2}}{r\br{r}}\bigg)d\br{z}d\br{r}\bigg|&\leq\iint_{A_{3}}\frac{\br{r}^{(3d)/2-3}}{r^{d/2-1}}\cdot\frac{1}{[(r-\br{r})^{2}+(z-\br{z})^{2}]^{1/2}}d\br{z}d\br{r}\\
			&\lesssim_{d}
			\iint_{A_{3}}\frac{\br{r}^{d-2}}{[(r-\br{r})^{2}+(z-\br{z})^{2}]^{1/2}}d\br{z}d\br{r}\\
			&\lesssim_{d}
			r^{d-2}\iint_{A_{3}}\frac{1}{[(r-\br{r})^{2}+(z-\br{z})^{2}]^{1/2}}d\br{z}d\br{r}.
		\end{split}
	\end{equation*}
	Here, note that we can estimate the integral in the last term as
	$$ \iint_{A_{3}}\frac{1}{[(r-\br{r})^{2}+(z-\br{z})^{2}]^{1/2}}d\br{z}d\br{r}\leq\int_{B_{\alp}}\frac{1}{|y|}dy\lesssim
	\alp, $$
	where $ B_{\alp}:=\lbrace y\in\bbR^{2} : |y|<\alp\rbrace $ is the disk whose measure in $ \bbR^{2} $ is the same as the measure of $ A_{3} $ in $ \Pi $. 
	In addition, we get
	\begin{equation*}
		\begin{split}
			|A_{3}|&=\bigg|\calA\cap\bigg\lbrace(\br{r},\br{z})\in\Pi : |r-\br{r}|\leq\frac{r}{2},\;|z-\br{z}|\leq \frac{r}{2}\bigg\rbrace\bigg|=\iint_{\substack{|r-\br{r}|\leq r/2, \\ |z-\br{z}|\leq r/2}}\xi(\br{r},\br{z})d\br{z}d\br{r}\\
			&=\iint_{\substack{|r-\br{r}|\leq r/2, \\ |z-\br{z}|\leq r/2}}\frac{\br{r}^{2d-3}}{\br{r}^{2d-3}}\cdot\xi(\br{r},\br{z})d\br{z}d\br{r}\lesssim_{d}
			\frac{1
			}{r^{2d-3}}\cdot\nrm{r^{2d-3}\xi}_{L^{1}(\Pi)}.
		\end{split}
	\end{equation*}
	From this, we obtain
	\begin{equation*}
		\begin{split}
			\iint_{A_{3}}\frac{1}{[(r-\br{r})^{2}+(z-\br{z})^{2}]^{1/2}}d\br{z}d\br{r}&\lesssim
			\alp
			\lesssim
			|A_{3}|^{1/2}\lesssim_{d,\xi}
			\frac{1
			}{r^{d-3/2}}.
		\end{split}
	\end{equation*}
	This gives us
	$$ \bigg|\iint_{A_{3}}\frac{\br{r}^{(3d)/2-4}(z-\br{z})}{r^{d/2}}F_{d}'\bigg(\frac{(r-\br{r})^{2}+(z-\br{z})^{2}}{r\br{r}}\bigg)d\br{z}d\br{r}\bigg|\lesssim_{d,\xi}
	r^{d-2}\cdot\frac{1
	}{r^{d-3/2}}=
	\frac{1
	}{\sqrt{r}}. $$
	Summing these up, we obtain \eqref{eq_simpleur}.
\end{proof}

\begin{rmk}\label{rmk_simpleu}
	Considering that $ \xi_{0} $ is a patch, as long as the solution of \eqref{eq_transport} exists, we can get the following time-independent pointwise decaying estimate of $ u^{r} $ from Proposition \ref{prop_simpleur}:
	\begin{equation}\label{eq_simpleur1}
		|u^{r}(t,r,z)|\lesssim_{d,\xi}
		\frac{1
		}{\sqrt{r}}\quad\text{for any}\quad r\geq1.
	\end{equation}
	From this, we have the ordinary differential inequality (see the definition of $ S(t) $ in \eqref{eq_defofSt})
	$$ S'(t)\lesssim_{d,\xi}
	\frac{1
	}{\sqrt{S(t)}}\quad\text{for any}\quad r=S(t). $$
	Solving this gives us, at least formally,
	$$ S(t)\lesssim (1+t)^{2/3}. $$

\end{rmk}

\section{Support growth of a single-signed vorticity for any $ d\geq3 $}\label{sec_singlesign}

\subsection{Refinement of the pointwise estimate of $ u^{r} $}\label{subsec_refinement}

Now we prove our main result. We assume the conditions in Theorem \ref{thm:3}, including the condition $ u_{0}\in H^{s}(\bbR^{d}) $ for some $ s>1+d/2 $. We follow the spirit of \cite{M94,ISG99}, using the dynamics of $ \omg $. 
Precisely, we take the approach of \cite{Maffei2001}, which is implicated in our Lemma \ref{lem_mafmar2}. 
Additionally, without loss of generality, we assume that $ \omg_{0} $ is non-negative. Then for each $ t\geq0 $ and $ r\geq0 $, we use the notation
$$ m_{r}(t):=\iint_{A_{r}}\omg(t,\br{r},\br{z})d\br{z}d\br{r},\quad A_{r}:=\lbrace (\br{r},\br{z})\in\Pi : \br{r}\geq r \rbrace, $$
where $ \omg $ is the solution of \eqref{eq_vortformNd} with the initial data $ \omg_{0} $.
\medskip

\noindent We introduce Lemma \ref{lem_mafmar} and Proposition \ref{prop_mafmar}. The proofs of Proposition \ref{prop_mafmar} and Theorem \ref{thm:3} are postponed after showing some technical lemmas that are needed to prove the proposition. 
From below, the constant $ C>0 $ may change line-by-line, and depends only on $ d $, $ \nrm{r\omg_{0}}_{L^{1}(\bbR^{d})} $, $ \nrm{r^{-(d-2)}\omg_{0}}_{L^{\ift}(\bbR^{d})} $ and $ S_{0} $, unless mentioned otherwise.

\medskip

\noindent The following lemma shows a \textit{pointwise} estimate of $ u^{r} $, which holds for any point in $ \Pi $.
\begin{lem}\label{lem_mafmar}
	There exists $ C>0 $ 
	such that for any $ t\geq0 $ and $ (r,z)\in\Pi $, we have
	\begin{equation}\label{eq_urclaim1}
		|u^{r}(t,r,z)|\leq C\bigg[\frac{1}{r^{d}}+(r^{d-2}+1)m_{r/2}(t)^{1/2}\bigg].
	\end{equation}
\end{lem}

\begin{proof}
	We recall the integral form \eqref{eq_urform} of $ u^{r} $. We split the integral range into $ E:=\lbrace(\br{r},\br{z})\in\Pi : |r-\br{r}|\leq r/2\rbrace $ and $ \Pi\setminus E $, and use the estimate \eqref{eq_Fdprime} in Lemma \ref{lem_Fd'est}. On $ \Pi\setminus E $, we have
	\begin{equation*}
		\begin{split}
			\bigg|\iint_{\Pi\setminus E}\frac{\br{r}^{d/2-2}(z-\br{z})}{r^{d/2}}F_{d}'\bigg(\frac{(r-\br{r})^{2}+(z-\br{z})^{2}}{r\br{r}}\bigg)&\omg(t,\br{r},\br{z})d\br{z}d\br{r}\bigg|\leq\iint_{\Pi\setminus E}\frac{\br{r}^{d/2-2}|z-\br{z}|}{r^{d/2}}\cdot\frac{(r\br{r})^{d/2+1}}{[(r-\br{r})^{2}+(z-\br{z})^{2}]^{d/2+1}}\omg(t,\br{r},\br{z})d\br{z}d\br{r}\\
			&\leq\iint_{\Pi\setminus E}\frac{r\br{r}^{d-1}}{[(r-\br{r})^{2}+(z-\br{z})^{2}]^{(d+1)/2}}\omg(t,\br{r},\br{z})d\br{z}d\br{r}
			\leq\frac{Cr}{r^{d+1}}
			\cdot\nrm{r^{d-1}\omg(t)}_{L^{1}(\Pi)}\\
			&\simeq_{d}\frac{C}{r^{d}}\cdot\nrm{r\omg(t)}_{L^{1}(\bbR^{d})}=\frac{C}{r^{d}}\cdot\nrm{r\omg_{0}}_{L^{1}(\bbR^{d})}.
		\end{split}
	\end{equation*}
	Here, we used the conservation $ \nrm{r\omg(t)}_{L^{1}(\bbR^{d})}=\nrm{r\omg_{0}}_{L^{1}(\bbR^{d})} $ 
	for a single-signed $ \omg $. On the other hand on $ E $, we have
	\begin{equation*}
		\begin{split}
			\bigg|\iint_{E}\frac{\br{r}^{d/2-2}(z-\br{z})}{r^{d/2}}F_{d}'\bigg(\frac{(r-\br{r})^{2}+(z-\br{z})^{2}}{r\br{r}}\bigg)\omg(t,\br{r},\br{z})d\br{z}d\br{r}\bigg|&\leq\iint_{E}\frac{\br{r}^{d/2-2}|z-\br{z}|}{r^{d/2}}\cdot\frac{r\br{r}}{(r-\br{r})^{2}+(z-\br{z})^{2}}\omg(t,\br{r},\br{z})d\br{z}d\br{r}\\
			&\leq\iint_{E}\bigg(\frac{\br{r}}{r}\bigg)^{d/2-1}\cdot\frac{\omg(t,\br{r},\br{z})}{[(r-\br{r})^{2}+(z-\br{z})^{2}]^{1/2}}d\br{z}d\br{r}\\
			&\leq C
			\iint_{E}\frac{\omg(t,\br{r},\br{z})}{[(r-\br{r})^{2}+(z-\br{z})^{2}]^{1/2}}d\br{z}d\br{r}.
		\end{split}
	\end{equation*}
	Now we split the integral range of the last term above into $ E\cap B_{k} $ and $ E\setminus B_{k} $, where
	$$ B_{k}=\lbrace(\br{r},\br{z})\in\Pi : [(r-\br{r})^{2}+(z-\br{z})^{2}]^{1/2}<k\rbrace,\quad k:=m_{r/2}(t)^{1/2}. $$
	Then we get
	\begin{equation*}
		\begin{split}
			\iint_{E\cap B_{k}}\frac{\omg(t,\br{r},\br{z})}{[(r-\br{r})^{2}+(z-\br{z})^{2}]^{1/2}}d\br{z}d\br{r}&=\iint_{E\cap B_{k}}\frac{\omg(t,\br{r},\br{z})}{\br{r}^{d-2}}\cdot\frac{\br{r}^{d-2}}{[(r-\br{r})^{2}+(z-\br{z})^{2}]^{1/2}}d\br{z}d\br{r}\\
			&\leq\bigg\|\frac{\omg(t)}{r^{d-2}}\bigg\|_{L^{\ift}(\Pi)}Cr^{d-2}
			\iint_{B_{k}}\frac{1}{[(r-\br{r})^{2}+(z-\br{z})^{2}]^{1/2}}d\br{z}d\br{r}\\
			&\leq\bigg\|\frac{\omg_{0}}{r^{d-2}}\bigg\|_{L^{\ift}(\bbR^{d})}Cr^{d-2}
			k=Cr^{d-2}m_{r/2}(t)^{1/2},
		\end{split}
	\end{equation*}
	and
	\begin{equation*}
		\begin{split}
			\iint_{E\setminus B_{k}}\frac{\omg(t,\br{r},\br{z})}{[(r-\br{r})^{2}+(z-\br{z})^{2}]^{1/2}}d\br{z}d\br{r}&\leq\frac{1}{k}\iint_{A_{r/2}}\omg(t,\br{r},\br{z})d\br{z}d\br{r}\\
			&=\frac{m_{r/2}(t)}{k}=m_{r/2}(t)^{1/2}.
		\end{split}
	\end{equation*}
	This proves \eqref{eq_urclaim1}.
\end{proof}
\noindent The following proposition says that we can shrink $ m_{r}(t) $ as much as we want it to be when $ r $ is depending on time.

\begin{prop}\label{prop_mafmar}
	Let $ q\geq d $ be an integer. 
	Then there exists $ C>0 $ 
	such that for any $ t\geq0 $ and $ r\geq0 $ that satisfies
	$$ r\geq C[(1+t)\ln(e+t)]^{1/(d+1)}, $$
	we have
	\begin{equation}\label{eq_rn+2}
		m_{r}(t)\leq\frac{1}{r^{q
		}}.
	\end{equation}
\end{prop}
\noindent 

\noindent 
The proof of Proposition \ref{prop_mafmar} is given in section \ref{sec_pfmain}.

\subsection{Some technical lemmas
}\label{subsec_techlem}

To prove Proposition \ref{prop_mafmar}, we need several lemmas. 
The first lemma below shows that $ m_{r_{2}}(t) $ is expressed in terms of $ m_{r_{1}}(s) $ for earlier times $ s\leq t $ and 
for $ r_{1}<r_{2} $. This type of lemma for $ d=3 $ is the heart of \cite{Maffei2001}.

\begin{lem}\label{lem_mafmar2}
	There exists $ C>0 $ 
	such that for any $ t\geq0 $ and $ r_{1}, r_{2}>0 $ satisfying $ S_{0}\leq r_{1}<r_{2} $, we have
	\begin{equation}\label{eq_mr2r1}
		\begin{split}
			m_{r_{2}}(t)&\leq J_{d}(r_{1},r_{2})\int_{0}^{t}m_{r_{1}}(s)ds,\quad J_{d}(r_{1},r_{2}):=\frac{Cr_{2}}{(r_{2}-r_{1})^{2}r_{1}^{d}}.
		\end{split}
	\end{equation}
\end{lem}
\begin{proof}
	First, we let $ \phi\in C_{c}^{\ift}(\bbR;[0,1]) $ be a smooth even function that satisfies $ \phi(0)=1 $, monotone on $ [0,\ift) $, and supported on $ [-1,1] $. 
	Then for each $ r_{1}, r_{2}>0 $ satisfying $ S_{0}\leq r_{1}<r_{2} $, we define a smooth monotone function $ \eta_{r_{1}r_{2}}:[0,\ift)\To[0,1] $, which is given as
	\begin{equation}\label{eq_etar1r2}
		\eta_{r_{1}r_{2}}(s):=\begin{cases}
			1-\phi(\frac{s-r_{1}}{r_{2}-r_{1}}) & \text{ if }s>r_{1}\\
			0 & \text{ if }0\leq s\leq r_{1}
		\end{cases}.
	\end{equation}
	Note that $ \eta_{r_{1}r_{2}} $ satisfies
	\begin{align}
		\eta_{r_{1}r_{2}}|_{r=0}&=0\label{eq_etabdry},\\
		|\eta_{r_{1}r_{2}}'(r)|&\leq\frac{C}{r_{2}-r_{1}}\quad\text{for any}\quad r>0,\label{eq_etaderbd}\\
		|\eta_{r_{1}r_{2}}'(r)-\eta_{r_{1}r_{2}}'(\br{r})|&\leq\frac{C|r-\br{r}|}{(r_{2}-r_{1})^{2}}\quad\text{for any}\quad r,\br{r}>0,\label{eq_etadermvt}
	\end{align}
	for some absolute constant $ C>0 $. Also for any $ t\geq0 $, we define
	$$ g(t):=\iint_{\Pi}\eta_{r_{1}r_{2}}(r)\omg(t,r,z)dzdr. $$
	Since $ S_{0}\leq r_{1}<r_{2} $, we have
	\begin{equation*}
		m_{r_{2}}(t)\leq g(t)\leq g(0)+\int_{0}^{t}|g'(s)|ds=\int_{0}^{t}|g'(s)|ds.
	\end{equation*}
	Our goal is to show
	\begin{equation}\label{eq_g'test}
		|g'(t)|\leq J_{d}(r_{1},r_{2})m_{r_{1}}(t).
	\end{equation}
	Once the above is shown, then the proof of \eqref{eq_mr2r1} is complete.
	
	\medskip
	
	\noindent To begin with, we take the time derivative of $ g $:
	\begin{equation*}
		\begin{split}
			g'(t)&=\iint_{\Pi}\eta_{r_{1}r_{2}}\rd_{t}\omg dzdr\\
			&=\iint_{\Pi}\eta_{r_{1}r_{2}}\bigg[-(u^{r}\rd_{r}+u^{z}\rd_{z})\omg+\frac{d-2}{r}u^{r}\omg\bigg]dzdr\\
			&=\iint_{\Pi}\bigg[[u^{r}\eta_{r_{1}r_{2}}'+\eta_{r_{1}r_{2}}(\underbrace{\rd_{r}u^{r}+\rd_{z}u^{z}}_{=-\frac{d-2}{r}u^{r}})
			]+\frac{d-2}{r}\eta_{r_{1}r_{2}} u^{r}\bigg]\omg dzdr=\iint_{\Pi}\eta_{r_{1}r_{2}}'u^{r}\omg dzdr.
		\end{split}
	\end{equation*}
	We used the fact that $ \omg $ solves \eqref{eq_vortformNd} in the second equality. In the third one, we used integration by parts, where the boundary term vanishes because $ \omg $ is compactly supported and $ \eta_{r_{1}r_{2}}|_{r=0}=0 $. In the last one, we used the divergence-free condition
	$$ \nb\cdot u=\rd_{r}u^{r}+\rd_{z}u^{z}+\frac{d-2}{r}u^{r}=0. $$
	Now we have
	\begin{equation*}
		\begin{split}
			g'(t)&=\iint_{\Pi}\eta_{r_{1}r_{2}}'(r)u^{r}(t,r,z)\omg(t,r,z)dzdr\\
			&\simeq_{d}\iint_{\Pi}\iint_{\Pi}\frac{\br{r}^{d/2-2}(z-\br{z})}{r^{d/2}}F_{d}'\bigg(\frac{(r-\br{r})^{2}+(z-\br{z})^{2}}{r\br{r}}\bigg)\eta_{r_{1}r_{2}}'(r)\omg(t,\br{r},\br{z})\omg(t,r,z)d\br{z}d\br{r}dzdr.
		\end{split}
	\end{equation*}
	Then symmetrizing the above integral, we have
	\begin{equation}\label{eq_g'tsymm}
		\begin{split}
			g'(t)
			&\simeq_{d}\iint_{\Pi}\iint_{\Pi}[\br{r}^{d-2}\eta_{r_{1}r_{2}}'(r)-r^{d-2}\eta_{r_{1}r_{2}}'(\br{r})]\frac{z-\br{z}}{(r\br{r})^{d/2}}F_{d}'\bigg(\frac{(r-\br{r})^{2}+(z-\br{z})^{2}}{r\br{r}}\bigg)\omg(t,\br{r},\br{z})\omg(t,r,z)d\br{z}d\br{r}dzdr.
		\end{split}
	\end{equation}
	We denote the kernel of the above integral as
	\begin{equation*}
		\begin{split}
			L(r,z,\br{r},\br{z})&:=[\br{r}^{d-2}\eta_{r_{1}r_{2}}'(r)-r^{d-2}\eta_{r_{1}r_{2}}'(\br{r})]\frac{z-\br{z}}{(r\br{r})^{d/2}}F_{d}'\bigg(\frac{(r-\br{r})^{2}+(z-\br{z})^{2}}{r\br{r}}\bigg),
		\end{split}
	\end{equation*}
	and use the notation as an abbreviation: $$ \omg=\omg(t,r,z)dzdr,\quad \br{\omg}=\omg(t,\br{r},\br{z})d\br{z}d\br{r}. $$ Then note that $ L(r,z,\br{r},\br{z}) $ is symmetric, and that $ L(r,z,\br{r},\br{z})=0 $ if $ r\in[r_{1},r_{2}]^{C} $ and $ \br{r}\in[r_{1},r_{2}]^{C} $, because of $ \eta_{r_{1}r_{2}}\equiv0 $ on $ [r_{1},r_{2}]^{C} $, which means
	\begin{equation*}
		\begin{split}
			\supp L
			&\subset\lbrace (r,z,\br{r},\br{z})\in\Pi\times\Pi : r_{1}\leq r\leq r_{2}\rbrace\cup\lbrace (r,z,\br{r},\br{z})\in\Pi\times\Pi : r_{1}\leq \br{r}\leq r_{2} \rbrace.
		\end{split}
	\end{equation*}
	From this, we have
	\begin{equation*}
		\begin{split}
			|g'(t)|&\lesssim_{d}\iint_{\Pi}\iint_{\Pi}|L(r,z,\br{r},\br{z})|\br{\omg}\omg\\
			&\leq\iint_{r_{1}\leq r\leq r_{2}}\iint_{\Pi}|L(r,z,\br{r},\br{z})|\br{\omg}\omg+\iint_{\Pi}\iint_{r_{1}\leq \br{r}\leq r_{2}}|L(r,z,\br{r},\br{z})|\br{\omg}\omg\\
			&=2\iint_{r_{1}\leq r\leq r_{2}}\iint_{\Pi}|L(r,z,\br{r},\br{z})|\br{\omg}\omg.
		\end{split}
	\end{equation*}
	Now we split $ L=L_{1}+L_{2} $, where
	\begin{equation*}
		\begin{split}
			L_{1}(r,z,\br{r},\br{z})&:=\br{r}^{d-2}[\eta_{r_{1}r_{2}}'(r)-\eta_{r_{1}r_{2}}'(\br{r})]\frac{z-\br{z}}{(r\br{r})^{d/2}}F_{d}'\bigg(\frac{(r-\br{r})^{2}+(z-\br{z})^{2}}{r\br{r}}\bigg),\\
			L_{2}(r,z,\br{r},\br{z})&:=(\br{r}^{d-2}-r^{d-2})\eta_{r_{1}r_{2}}'(\br{r})\frac{z-\br{z}}{(r\br{r})^{d/2}}F_{d}'\bigg(\frac{(r-\br{r})^{2}+(z-\br{z})^{2}}{r\br{r}}\bigg).
		\end{split}
	\end{equation*}
	From this, we get
	\begin{equation*}
		\begin{split}
			|g'(t)|&\lesssim_{d}
			\iint_{r_{1}\leq r\leq r_{2}}\iint_{\Pi}|L(r,z,\br{r},\br{z})|\br{\omg}\omg\\
			&\leq 
			\underbrace{\iint_{r_{1}\leq r\leq r_{2}}\iint_{\Pi}|L_{1}(r,z,\br{r},\br{z})|\br{\omg}\omg}_{=:(A)}+\underbrace{\iint_{r_{1}\leq r\leq r_{2}}\iint_{\Pi}|L_{2}(r,z,\br{r},\br{z})|\br{\omg}\omg}_{=:(B)}.
		\end{split}
	\end{equation*}
	To estimate the term $ |L_{1}| $, we can consider the following two cases:\\
	
	\noindent\textbf{Case I.} $ r_{1}\leq r\leq r_{2},\;|r-\br{r}|>r/2 $.
	
	\medskip
	\noindent In this case, 
	the term $ |L_{1}| $ is estimated as
	\begin{equation*}
		\begin{split}
			|L_{1}(r,z,\br{r},\br{z})|&\lesssim_{d}
			\br{r}^{d-2}\cdot\frac{1
			}{(r_{2}-r_{1})^{2}}\cdot|r-\br{r}|\cdot\frac{|z-\br{z}|}{(r\br{r})^{d/2}}\cdot\frac{(r\br{r})^{d/2+1}}{[(r-\br{r})^{2}+(z-\br{z})^{2}]^{d/2+1}}\\
			&\lesssim_{d}
			\frac{1
			}{(r_{2}-r_{1})^{2}}\cdot\frac{r\br{r}^{d-1}}{[(r-\br{r})^{2}+(z-\br{z})^{2}]^{d/2}}\\
			&\lesssim_{d}
			\frac{1
			}{(r_{2}-r_{1})^{2}}\cdot\frac{
			\br{r}^{d-1}}{r^{d-1}}\leq\frac{
			\br{r}^{d-1}}{(r_{2}-r_{1})^{2}r_{1}^{d-1}}.
		\end{split}
	\end{equation*}
	\textbf{Case II.} $ r_{1}\leq r\leq r_{2},\;|r-\br{r}|\leq r/2 $.
	
	\medskip
	\noindent Here, 
	we can estimate the term $ |L_{1}| $ as
	\begin{equation*}
		\begin{split}
			|L_{1}(r,z,\br{r},\br{z})|&\lesssim_{d}
			\br{r}^{d-2}\cdot\frac{1
			}{(r_{2}-r_{1})^{2}}\cdot|r-\br{r}|\cdot\frac{|z-\br{z}|}{(r\br{r})^{d/2}}\cdot\frac{r\br{r}}{(r-\br{r})^{2}+(z-\br{z})^{2}}\\
			&\leq\frac{1
			}{(r_{2}-r_{1})^{2}}\cdot\bigg(\frac{\br{r}}{r}\bigg)^{d/2-1}
			\lesssim_{d}
			\frac{
				\br{r}^{d-1}}{(r_{2}-r_{1})^{2}r_{1}^{d-1}}.
		\end{split}
	\end{equation*}
	Thus, we obtain
	\begin{equation}\label{eq_A1}
		\begin{split}
			(A)&\lesssim_{d}
			\frac{1
			}{(r_{2}-r_{1})^{2}r_{1}^{d-1}}\iint_{r_{1}\leq r\leq r_{2}}\iint_{\Pi}\br{r}^{d-1}\br{\omg}\omg\\
			&\leq\frac{1
			}{(r_{2}-r_{1})^{2}r_{1}^{d-1}}\cdot\underbrace{\nrm{r^{d-1}\omg}_{L^{1}(\Pi)}}_{=\nrm{r^{d-1}\omg_{0}}_{L^{1}(\Pi)}}\cdot m_{r_{1}}(t).
		\end{split}
	\end{equation}

	\noindent Now let us focus our attention to $ (B) $. Note that $ L_{2}(r,z,\br{r},\br{z})=0 $ if $ \br{r}\in[r_{1},r_{2}]^{C} $. From this, we have
	\begin{equation*}
		\begin{split}
			(B)&=\iint_{r_{1}\leq r\leq r_{2}}\iint_{r_{1}\leq \br{r}\leq r_{2}}|L_{2}(r,z,\br{r},\br{z})|\br{\omg}\omg.
		\end{split}
	\end{equation*}
	Similarly, to estimate the term $ |L_{2}| $, let us consider these two cases:\\
	
	\noindent\textbf{Case I'.} $ r_{1}\leq r\leq r_{2},\;r_{1}\leq\br{r}\leq r_{2},\;|r-\br{r}|>r/2 $.
	
	\medskip
	\noindent In this case, the term $ |L_{2}| $ is estimated as
	\begin{equation*}
		\begin{split}
			|L_{2}(r,z,\br{r},\br{z})|&\lesssim_{d}
			|r-\br{r}|\cdot\underbrace{(\br{r}^{d-3}+\br{r}^{d-4}r+\cdots+\br{r}r^{d-4}+r^{d-3})}_{(d-2)\text{ -terms}}\cdot\frac{1
			}{r_{2}-r_{1}}\cdot\frac{|z-\br{z}|}{(r\br{r})^{d/2}}\cdot\frac{(r\br{r})^{d/2+1}}{[(r-\br{r})^{2}+(z-\br{z})^{2}]^{d/2+1}}\\
			&\lesssim_{d}\frac{1}{r_{2}-r_{1}}\cdot\bigg(\frac{\br{r}^{d-2}}{r^{d-1}}+\frac{\br{r}^{d-3}}{r^{d-2}}+\cdots+\frac{\br{r}^{2}}{r^{3}}+\frac{\br{r}}{r^{2}}\bigg)=\frac{\br{r}^{d-1}}{r_{2}-r_{1}}\cdot\bigg(\frac{1}{r^{d-1}\br{r}}+\frac{1}{r^{d-2}\br{r}^{2}}+\cdots+\frac{1}{r^{3}\br{r}^{d-3}}+\frac{1}{r^{2}\br{r}^{d-2}}\bigg)\\
			&
			\lesssim_{d}
			\frac{
				\br{r}^{d-1}}{(r_{2}-r_{1})r_{1}^{d}}.
		\end{split}
	\end{equation*}
	\textbf{Case II'.} $ r_{1}\leq r\leq r_{2},\;r_{1}\leq\br{r}\leq r_{2},\;|r-\br{r}|\leq r/2 $.
	
	\medskip
	\noindent Here, we can estimate the term $ |L_{2}| $ as
	\begin{equation*}
		\begin{split}
			|L_{2}(r,z,\br{r},\br{z})|&\lesssim_{d}
			|r-\br{r}|\cdot(\br{r}^{d-3}+\br{r}^{d-4}r+\cdots+\br{r}r^{d-4}+r^{d-3})\cdot\frac{1
			}{r_{2}-r_{1}}\cdot\frac{|z-\br{z}|}{(r\br{r})^{d/2}}\cdot\frac{r\br{r}}{(r-\br{r})^{2}+(z-\br{z})^{2}}\\
			&\lesssim_{d}\frac{1}{r_{2}-r_{1}}\cdot\frac{r^{d/2-2}}{\br{r}^{d/2-1}}\lesssim_{d}
			\frac{
				\br{r}^{d-1}}{(r_{2}-r_{1})r_{1}^{d}}.
		\end{split}
	\end{equation*}
	Hence, we obtain
	\begin{equation}\label{eq_B1}
		\begin{split}
			(B)&\leq\frac{C}{(r_{2}-r_{1})r_{1}^{d}}\iint_{r_{1}\leq r\leq r_{2}}\iint_{r_{1}\leq\br{r}\leq r_{2}}\br{r}^{d-1}\br{\omg}\omg\\
			&\leq\frac{C}{(r_{2}-r_{1})r_{1}^{d}}\cdot\nrm{r^{d-1}\omg_{0}}_{L^{1}(\Pi)}\cdot m_{r_{1}}(t).
		\end{split}
	\end{equation}

	\noindent Finally, summing up \eqref{eq_A1} and \eqref{eq_B1}, we have
	\begin{equation*}
		\begin{split}
			|g'(t)|&\leq C\nrm{r^{d-1}\omg_{0}}_{L^{1}(\Pi)}\cdot\bigg[\frac{1}{(r_{2}-r_{1})^{2}r_{1}^{d-1}}+\frac{1}{(r_{2}-r_{1})r_{1}^{d}}\bigg]\cdot m_{r_{1}}(t)\\
			&=C\nrm{r\omg_{0}}_{L^{1}(\bbR^{d})}\cdot\frac{r_{2}}{(r_{2}-r_{1})^{2}r_{1}^{d}}\cdot m_{r_{1}}(t).
		\end{split}
	\end{equation*}
\end{proof}
\noindent The second lemma is an estimate of $ m_{r}(t) $ in terms of only $ r $ and $ t $, which is a direct application of the previous lemma.
\begin{lem}\label{lem_techlem3}
	There exists $ C>0 $ 
	such that for any integer $ p\geq1 $, any $ t\geq0 $ and any $ r\geq 2S_{0} $, we have
	\begin{equation}\label{eq_mrtkk}
		m_{r}(t)\leq \frac{C^{p}p^{p}t^{p}}{r^{p(d+1)+d-1}}.
	\end{equation}
\end{lem}
\begin{proof}
	We denote $ \alp_{j}:=r(1-\frac{j}{2p}) $ for $ j=0,1,\cdots,p $. Then note that we have $ \alp_{0}=r $, $ \alp_{p}=r/2 $, and for each $ j=1,\cdots,p $, by recalling the definition of $ J_{d}(r_{1},r_{2}) $ in Lemma \ref{lem_mafmar2}, we have
	\begin{equation*}
		\begin{split}
			J_{d}(\alp_{j},\alp_{j-1})&\leq\frac{Cr(1-\frac{j-1}{2p})}{(\frac{r}{2p})^{2}r^{d}(1-\frac{j}{2p})^{d}}
			\leq\frac{Cp^{2}}{r^{d+1}}.
		\end{split}
	\end{equation*}
	Using this and the inequality \eqref{eq_mr2r1} from Lemma \ref{lem_mafmar2}, for each $ j=1,\cdots,p $, we get
	\begin{equation}\label{eq_malpj}
		\begin{split}
			m_{\alp_{j-1}}(t)&\leq J_{d}(\alp_{j},\alp_{j-1})\int_{0}^{t}m_{\alp_{j}}(s)ds\leq\frac{Cp^{2}}{r^{d+1}}\int_{0}^{t}m_{\alp_{j}}(s)ds.
		\end{split}
	\end{equation}
	From this, we get
	\begin{equation*}
		\begin{split}
			m_{r}(t)&=m_{\alp_{0}}(t)\leq\frac{Cp^{2}}{r^{d+1}}\int_{0}^{t}m_{\alp_{1}}(s_{1})ds_{1}\\
			&\leq\bigg(\frac{Cp^{2}}{r^{d+1}}\bigg)^{2}\int_{0}^{t}\int_{0}^{s_{1}}m_{\alp_{2}}(s_{2})ds_{2}ds_{1}\leq\cdots\leq\bigg(\frac{Cp^{2}}{r^{d+1}}\bigg)^{p-1}\int_{0}^{t}\int_{0}^{s_{1}}\cdots\int_{0}^{s_{p-3}}\int_{0}^{s_{p-2}}m_{\alp_{p}}(s_{p-1})
			ds_{p-1}ds_{p-2}\cdots ds_{2}ds_{1}\\
			&\leq\bigg(\frac{Cp^{2}}{r^{d+1}}\bigg)^{p}\int_{0}^{t}\int_{0}^{s_{1}}\cdots\int_{0}^{s_{p-3}}\int_{0}^{s_{p-2}}\int_{0}^{s_{p-1}}m_{r/2}(s_{p})
			ds_{p}ds_{p-1}ds_{p-2}\cdots ds_{2}ds_{1}.
		\end{split}
	\end{equation*}
	Then using the estimate
	\begin{equation*}
		\begin{split}
			m_{r/2
			}(s_{p})&=\iint_{
			\br{r}\geq r/2
			}\omg(s_{p},\br{r},\br{z})d\br{z}d\br{r}\\
			&\leq\bigg(\frac{r}{2}\bigg)^{-(d-1)}
			\iint_{\Pi}\br{r}^{d-1}\omg(s_{p},\br{r},\br{z})d\br{z}d\br{r}=\bigg(\frac{r}{2}\bigg)^{-(d-1)}\cdot\nrm{r^{d-1}\omg_{0}}_{L^{1}(\Pi)},
		\end{split}
	\end{equation*}
	we have
	\begin{equation*}
		\begin{split}
			m_{r}(t)&\leq\frac{C^{p}p^{2p}}{r^{p(d+1)}}\cdot\bigg(\frac{r}{2}\bigg)^{-(d-1)}\cdot\nrm{r^{d-1}\omg_{0}}_{L^{1}(\Pi)}
			\int_{0}^{t}\int_{0}^{s_{1}}\cdots\int_{0}^{s_{p-3}}\int_{0}^{s_{p-2}}\int_{0}^{s_{p-1}}
			1ds_{p}ds_{p-1}ds_{p-2}\cdots ds_{2}ds_{1}.
		\end{split}
	\end{equation*}
	Finally, from the computation of the following term
	\begin{equation*}
		\begin{split}
			\int_{0}^{t}\int_{0}^{s_{1}}\cdots\int_{0}^{s_{p-3}}\int_{0}^{s_{p-2}}\int_{0}^{s_{p-1}}1ds_{p}ds_{p-1}ds_{p-2}\cdots ds_{2}ds_{1}&=\int_{0}^{t}\int_{0}^{s_{1}}\cdots\int_{0}^{s_{p-3}}\frac{(s_{p-2})^{2}}{2}ds_{p-2}\cdots ds_{2}ds_{1}\\
			&=\cdots=\int_{0}^{t}\frac{(s_{1})^{p-1}}{(p-1)!}ds_{1}=\frac{t^{p}}{p!},
		\end{split}
	\end{equation*}
	and the simple inequality $ \frac{p^{p}}{p!}\leq e^{p} $, we obtain
	\begin{equation*}
		\begin{split}
			m_{r}(t)&\leq\frac{C^{p}p^{2p}}{r^{p(d+1)+d-1}}\cdot\frac{t^{p}}{p!}\leq\frac{C^{p}p^{2p}}{r^{p(d+1)+d-1}}\cdot\frac{e^{p}t^{p}}{p^{p}}\\
			&\leq\frac{C^{p}p^{p}t^{p}}{r^{p(d+1)+d-1}}.
		\end{split}
	\end{equation*}
\end{proof}

\noindent The final lemma is about some facts from 
elementary computations. The first computation, where we roughly have $ k\approx\tau\ln y $, enables us to get the decay rate of $ m_{r} $ as fast as we want it to be. The second computation provides the necessary condition that $ r $ must satisfy in order to obtain the decay of $ m_{r} $.
\begin{lem}\label{lem_factanalysis}
	(a) For any integer $ \tau\geq1 $, there exist $ C=C(\tau)>0 $ and $ y_{0}=y_{0}(\tau)>1 $ such that for any $ y\geq y_{0} $, there exists an integer $ k=k(\tau,y)\geq1 $ that satisfies
	\begin{equation}\label{eq_xkl}
		y^{\tau/k}k\leq C\ln(e+y).
	\end{equation}
	(b) For any integer $ d\geq3 $ and any $ B>0 $, there exists $ A=A(d,B)>0 $ such that for any $ x\geq0 $ and any $ y\geq0 $ satisfying
	$$ y\geq A[(1+x)\ln(e+x)]^{1/(d+1)}, $$
	we have
	\begin{equation}\label{eq_yxlny}
		y\geq B[(1+x)\ln(e+y)]^{1/(d+1)}.
	\end{equation}
\end{lem}

\begin{proof}
	(a) We let $ \tau\geq1 $ be an integer, and for any $ y>1 $ and any $ s>0 $, we define
	$$ f_{y}(s):=y^{\tau/s}s-C\ln(e+y), $$
	where $ C>0 $ will be taken later. 
	Note that $ f_{y} $ is increasing for all $ s>\tau\ln y $:
	$$ f_{y}'(s)=y^{\tau/s}\bigg(1-\frac{\tau\ln y}{s}\bigg)>0\quad\text{for all}\quad s>\tau\ln y. $$
	Now we take $ C:=2e(\tau+1) $. 
	Then we can take $ y_{0}=y_{0}(\tau)>1 $ that satisfies
	\begin{equation}\label{eq_y0}
		f_{y}(\tau\ln y+1)\leq0\quad\text{for any}\quad y\geq y_{0},
	\end{equation}
	because we have
	\begin{equation*}
		\begin{split}
			f_{y}(\tau\ln y+1)&=y^{1/(\ln y+1/\tau)}(\tau\ln y+1)-C\ln(e+y)\\
			&\leq[\underbrace{y^{1/\ln y}}_{=e}(\tau+1)-C]\ln(e+y)
			\To-\ift\quad\text{as}\quad y\To\ift.
		\end{split}
	\end{equation*}
	Thus, from \eqref{eq_y0} and the fact that $ f_{y} $ is increasing on $ [\tau\ln y, \ift) $, 
	for any $ y\geq y_{0} $, we have
	$$ f_{y}(s)\leq0\quad\text{for any}\quad s\in[\tau\ln y,\tau\ln y+1]. $$
	Finally, we choose any positive integer $ k=k(\tau,y)\in [\tau\ln y,\tau\ln y+1] $ to obtain \eqref{eq_xkl}:
	$$ f_{y}(k)=y^{\tau/k}k-C\ln(e+y)\leq0. $$		
	(b) We let $ d\geq3 $. We first claim that there exists an absolute constant $ C_{0}>0 $ such that for any $ \alp>0 $ and any $ x\geq0 $, we have
	\begin{equation}\label{eq_aclaim}
		\frac{\ln\big[e+\alp[(1+x)\ln(e+x)]^{1/(d+1)}\big]}{\ln(e+x)}\leq C_{0}(1+\sqrt{\alp}).
	\end{equation}
	Once this claim is shown, then for any $ B>0 $, we take any $ A=A(d,B)>0 $ satisfying
	\begin{equation}\label{eq_BArel}
		B=\frac{A}{C_{0}^{1/(d+1)}(1+\sqrt{A})^{1/(d+1)}},
	\end{equation}
	where the above claim \eqref{eq_aclaim} gives us for any $ x\geq0 $,
	\begin{equation}\label{eq_BAlnrel}
		B^{d+1}=\frac{A^{d+1}}{C_{0}(1+\sqrt{A})}\leq\frac{A^{d+1}\ln(e+x)}{\ln\big[e+A[(1+x)\ln(e+x)]^{1/(d+1)}\big]}.
	\end{equation}
	Then for any $ y\geq A[(1+x)\ln(e+x)]^{1/(d+1)} $, using the fact that $ y^{d+1}/\ln(e+y) $ is monotonically increasing in $ y>0 $, we have
	\begin{equation*}
		\frac{y^{d+1}}{\ln(e+y)}\geq\frac{A^{d+1}(1+x)\ln(e+x)}{\ln\big[e+A[(1+x)\ln(e+x)]^{1/(d+1)}\big]}\geq B^{d+1}(1+x),
	\end{equation*}
	which proves \eqref{eq_yxlny}.
	
	\medskip
	
	\noindent Now let us prove our claim \eqref{eq_aclaim}. 
	First, note that there exists an absolute constant $ c>1 $ that satisfies
	$$ c\sqrt{1+x}\geq[(1+x)\ln(e+x)]^{1/(d+1)}\quad\text{for any}\quad x\geq0. $$
	Now we let $ \alp>0 $ and $ x\geq0 $. Then we have
	\begin{equation*}
		\frac{\ln\big[e+\alp[(1+x)\ln(e+x)]^{1/(d+1)}\big]}{\ln(e+x)}\leq\frac{\ln\big(e+c\alp\sqrt{1+x}\big)}{\ln(e+x)}.
	\end{equation*}
	Also, note that there exists $ x_{0}=x_{0}(\alp)>0 $ that satisfies
	$$ x_{0}=c\alp\sqrt{1+x_{0}}. $$
	In particular, we have
	$$ x_{0}=\frac{(c\alp)^{2}+c\alp\sqrt{(c\alp)^{2}+4}}{2}. $$
	Then if $ x\geq x_{0} $, we have $ x\geq c\alp\sqrt{1+x} $, which gives us
	\begin{equation}\label{eq_xgeqx0}
		\frac{\ln\big(e+c\alp\sqrt{1+x}\big)}{\ln(e+x)}\leq\frac{\ln(e+x)}{\ln(e+x)}=1.
	\end{equation}
	If $ x<x_{0} $, then we get
	\begin{equation*}
		\begin{split}
			\frac{\ln\big(e+c\alp\sqrt{1+x}\big)}{\ln(e+x)}&\leq\ln\big(e+c\alp\sqrt{1+x_{0}}\big)=\ln(e+x_{0})\\
			&=\ln\bigg(e+\frac{(c\alp)^{2}+c\alp\sqrt{(c\alp)^{2}+4}}{2}\bigg)\leq\ln\big(e+(c\alp+2)^{2}\big)\\
			&\leq 2\ln(c\alp+2+\sqrt{e})\leq 2\ln\big(e^{C'}(\alp+1)\big)\\
			&\leq2\big(C'+\ln(\alp+1)\big),
		\end{split}
	\end{equation*}
	where $ C'=\ln\max\lbrace c,2+\sqrt{e}\rbrace $. Then using $ \ln(\alp+1)\leq\sqrt{\alp} $, we obtain
	\begin{equation}\label{eq_xlx0}
		\begin{split}
			\frac{\ln\big(e+c\alp\sqrt{1+x}\big)}{\ln(e+x)}&\leq2\big(C'+\ln(\alp+1)\big)\\
			&\leq2(C'+\sqrt{\alp})\leq C_{0}(1+\sqrt{\alp}),
		\end{split}
	\end{equation}
	where $ C_{0}=2\max\lbrace C',1\rbrace $. Hence, from \eqref{eq_xgeqx0} and \eqref{eq_xlx0}, we have \eqref{eq_aclaim}:
	$$ \sup_{x\geq0}\frac{\ln\big(e+c\alp\sqrt{1+x}\big)}{\ln(e+x)}\leq C_{0}(1+\sqrt{\alp}). $$
\end{proof}

\section{Proof of the main result}\label{sec_pfmain}

We present the proof of Proposition \ref{prop_mafmar} and Theorem \ref{thm:3} in this section.

\subsection{Proof of Proposition \ref{prop_mafmar}}\label{subsec_proppf}

Using lemmas from section \ref{sec_singlesign}, let us prove Proposition \ref{prop_mafmar}.

\begin{proof}[Proof of Proposition \ref{prop_mafmar}]
	We let $ d\geq3 $, $ q\geq d $ be integers, and $ \omg_{0} $ satisfy the conditions from Theorem \ref{thm:3}. Then from Lemma \ref{lem_techlem3}, there exists $ C'>0 $ such that for any $ p\in\bbN $, any $ t\geq0 $, and any $ r\geq 2S_{0} $, we have \eqref{eq_mrtkk}:
	\begin{equation}\label{eq_mrtrp}
		m_{r}(t)\leq\frac{(C')^{p}p^{p}t^{p}}{r^{p(d+1)+d-1}}.
	\end{equation}
	Now we fix $ \tau:=q-d+1 $. Then from Lemma \ref{lem_factanalysis} (a), there exist $ C''
	>0 $ and $ y_{0}
	>1 $ such that for any $ r\geq y_{0} $, there exists an integer $ k=k(r)
	\geq1 $ that satisfies \eqref{eq_xkl} with $ y=r $:
	$$ r^{\tau/k}k\leq C''\ln(e+r). $$
	We can rewrite this as
	\begin{equation}\label{eq_rqkl}
		k^{k}\leq\frac{[C''\ln(e+r)]^{k}}{r^{q-d+1}}.
	\end{equation}
	Then we set 
	\begin{equation*}
		B:=(C'C'')^{1/(d+1)}.
	\end{equation*}
	Then from Lemma \ref{lem_factanalysis} (b), there exists $ A
	>0 $ such that for any $ t\geq0 $ and any $ r\geq A[(1+t)\ln(e+t)]^{1/(d+1)} $, we have \eqref{eq_yxlny} with $ x=t $ and $ y=r $:
	\begin{equation*}
		r\geq B[(1+t)\ln(e+r)]^{1/(d+1)}.
	\end{equation*}
	Rewriting this with respect to $ t $, we get
	\begin{equation}\label{eq_rBter}
		\begin{split}
			t&\leq 1+t\\
			&\leq\frac{1}{B^{d+1}}\cdot\frac{r^{d+1}}{\ln(e+r)}\leq\frac{1}{C'C''}\cdot\frac{r^{d+1}}{\ln(e+r)}.
		\end{split}
	\end{equation}
	
	\noindent Finally, we take $ C:=\max\lbrace 2S_{0}, y_{0}, A\rbrace>0 $. Then for any $ t\geq0 $ and any $ r\geq C[(1+t)\ln(e+t)]^{1/(d+1)} $, summing up \eqref{eq_mrtrp} with the choice $ p=k(r) $, \eqref{eq_rBter}, and \eqref{eq_rqkl}, we obtain \eqref{eq_rn+2}:
	\begin{equation*}
		\begin{split}
			m_{r}(t)&\leq\frac{(C')^{k}k^{k}t^{k}}{r^{k(d+1)+d-1}}\\
			&\leq\frac{(C')^{k}}{r^{k(d+1)+d-1}}\cdot\frac{[C''\ln(e+r)]^{k}}{r^{q-d+1}}\cdot\frac{1}{(C'C'')^{k}}\cdot\frac{r^{k(d+1)}}{[\ln(e+r)]^{k}}=\frac{1}{r^{q}}.
		\end{split}
	\end{equation*}
\end{proof}

\subsection{Proof of Theorem \ref{thm:3}}\label{subsec_pfthm}

Finally, we finish this paper by showing Theorem \ref{thm:3}.

\begin{proof}[Proof of Theorem \ref{thm:3}]
	It suffices to prove the estimate \eqref{eq_StestinRd};
	$$ S(t)\leq C[(1+t)\ln(e+t)]^{1/(d+1)}\quad\text{for any}\quad t\geq0, $$
	thanks to Remark \ref{rmk_actualest} and the BKM criterion. We let $ t\geq0 $. First, from Lemma \ref{lem_mafmar}, there exists $ C'>0 $ 
	such that for any $ t\geq0 $ and $ (r,z)\in\Pi $, we have \eqref{eq_urclaim1}:
	\begin{equation}\label{eq_urgeneral}
		|u^{r}(t,r,z)|\leq C'\bigg[\frac{1}{r^{d}}+(r^{d-2}+1)m_{r/2}(t)^{1/2}\bigg].
	\end{equation}
	In addition, choosing $ q=4d-4 $ 
	from Proposition \ref{prop_mafmar}, we can take $ C''>0 $ such that for any \linebreak 
	$ r\geq C''[(1+t)\ln(e+t)]^{1/(d+1)} $, we have
	\begin{equation}\label{eq_mr2min}
		m_{r/2}(t)\leq\bigg(\frac{2}{r}\bigg)^{4d-4}.
	\end{equation}
	Thus, combining \eqref{eq_urgeneral} and \eqref{eq_mr2min}, we can choose $ C_{0}:=\max\lbrace C'', C'(1+2\cdot2^{2d-2})\rbrace>0 $ such that for any 
	\linebreak $ r\geq C_{0}[(1+t)\ln(e+t)]^{1/(d+1)} $, we have
	\begin{equation*}
		\begin{split}
			|u^{r}(t,r,z)|&\leq C'\bigg[\frac{1}{r^{d}}+(r^{d-2}+1)m_{r/2}(t)^{1/2}\bigg]\\
			&\leq C'\bigg[\frac{1}{r^{d}}+2r^{d-2}\bigg(\frac{2}{r}\bigg)^{2d-2}
			\bigg]=\frac{C_{0}}{r^{d}}.
		\end{split}
	\end{equation*}
	Here, we denote $ \calF(t):=[(1+t)\ln(e+t)]^{1/(d+1)} $. Then note that for any $ r\geq C_{0}\calF(t) $, we have
	\begin{equation*}
		\begin{split}
			|u^{r}(t,r,z)|&\leq \frac{C_{0}}{r^{d}}\leq\frac{1}{C_{0}^{d-1}\calF(t)^{d}}\\
			&\leq\frac{d+1}{C_{0}^{d-1}}\calF'(t).
		\end{split}
	\end{equation*}
	Then we take $ C:=\max\lbrace C_{0}, (d+1)/C_{0}^{d-1}, S_{0}\rbrace>0 $, which gives us for any 
	$ r\geq C\calF(t) $,
	\begin{equation}\label{eq_urRc't}
		|u^{r}(t,r,z)|\leq \frac{d+1}{C_{0}^{d-1}}\calF'(t)
		\leq [C\calF(t)]'.
	\end{equation}
	This proves \eqref{eq_StestinRd}; for any $ x\in\supp\omg_{0} $, we have
	\begin{equation}\label{eq_goalofthm3}
		\Phi_{t}^{r}(x)\leq C\calF(t),
	\end{equation}
	where $ \Phi_{t}(\cdot)=\Phi(t,\cdot) $ is the flow map that uniquely solves the flow map ODE
	\begin{equation}\label{eq_flowmapode}
		\frac{d}{dt}\Phi_{t}(x)=u\big(t,\Phi_{t}(x)\big),\quad \Phi_{0}(x)=x,
	\end{equation}
	since we have $ \Phi_{0}^{r}(x)\leq C $ and \eqref{eq_urRc't} implies
	$$ \bigg|\frac{d}{dt}\Phi_{t}^{r}(x)\bigg|=\big|u^{r}\big(t,\Phi_{t}(x)\big)\big|\leq [C\calF(t)]'\quad\text{if}\quad \Phi_{t}^{r}(x)=C\calF(t)\quad\text{happens for some}\quad t\geq0. $$
\end{proof}

\subsection*{Acknowledgments}

The author wishes to thank Prof. Kyudong Choi and Prof. In-Jee Jeong for providing helpful discussions and references in writing this paper.

\bibliographystyle{plain}
\bibliography{CJL}

\begin{thebibliography}{10}

\bibitem{BKM}
J.~T. Beale, T.~Kato, and A.~Majda.
\newblock Remarks on the breakdown of smooth solutions for the {$3$}-{D}
  {E}uler equations.
\newblock {\em Comm. Math. Phys.}, 94(1):61--66, 1984.

\bibitem{ChilGil2}
Stephen Childress and Andrew~D. Gilbert.
\newblock Eroding dipoles and vorticity growth for {E}uler flows in {$\Bbb
  R^3$}: the hairpin geometry as a model for finite-time blowup.
\newblock {\em Fluid Dyn. Res.}, 50(1):011418, 40, 2018.

\bibitem{ChilGil1}
Stephen Childress, Andrew~D. Gilbert, and Paul Valiant.
\newblock Eroding dipoles and vorticity growth for {E}uler flows in
  {$\Bbb{R}^3$}: axisymmetric flow without swirl.
\newblock {\em J. Fluid Mech.}, 805:1--30, 2016.

\bibitem{CD2019}
Kyudong Choi and Sergey Denisov.
\newblock On the {G}rowth of the {S}upport of {P}ositive {V}orticity for 2{D}
  {E}uler {E}quation in an {I}nfinite {C}ylinder.
\newblock {\em Comm. Math. Phys.}, 367(3):1077--1093, 2019.

\bibitem{CJ-axi}
Kyudong Choi and In-Jee Jeong.
\newblock On vortex stretching for anti-parallel axisymmetric flows.
\newblock {\em arXiv:2110.09079}.

\bibitem{CJLglobal22}
Kyudong Choi, In-Jee Jeong, and Deokwoo Lim.
\newblock Global regularity for some axisymmetric {E}uler flows in
  $\mathbb{R}^{d}$.
\newblock {\em Proc. Amer. Math. Soc., to appear, arXiv:2212.11461}.

\bibitem{Danaxi}
R.~Danchin.
\newblock Axisymmetric incompressible flows with bounded vorticity.
\newblock {\em Uspekhi Mat. Nauk}, 62(3(375)):73--94, 2007.

\bibitem{FeSv}
H.~Feng and V.~{\v{S}}ver{\'a}k.
\newblock On the {C}auchy problem for axi-symmetric vortex rings.
\newblock {\em Arch. Ration. Mech. Anal.}, 215:89--123, 2015.

\bibitem{GMT2023}
Stephen Gustafson, Evan Miller, and Tai-Peng Tsai.
\newblock Growth rates for anti-parallel vortex tube {E}uler flows in three and
  higher dimensions.
\newblock {\em preprint, arXiv:2303.12043}.

\bibitem{ILL2003}
D.~Iftimie, M.~C. Lopes~Filho, and H.~J. Nussenzveig~Lopes.
\newblock Large time behavior for vortex evolution in the half-plane.
\newblock {\em Comm. Math. Phys.}, 237(3):441--469, 2003.

\bibitem{ILL2007}
D.~Iftimie, M.~C. Lopes~Filho, and H.~J. Nussenzveig~Lopes.
\newblock Confinement of vorticity in two dimensional ideal incompressible
  exterior flow.
\newblock {\em Quart. Appl. Math.}, 65(3):499--521, 2007.

\bibitem{Iftimie1999}
Drago\c{s} Iftimie.
\newblock \'{E}volution de tourbillon \`a support compact.
\newblock In {\em Journ\'{e}es ``\'{E}quations aux {D}\'{e}riv\'{e}es
  {P}artielles'' ({S}aint-{J}ean-de-{M}onts, 1999)}, pages Exp. No. IV, 8.
  Univ. Nantes, Nantes, 1999.

\bibitem{ISG99}
Drago\c{s} Iftimie, Thomas~C. Sideris, and Pascal Gamblin.
\newblock On the evolution of compactly supported planar vorticity.
\newblock {\em Comm. Partial Differential Equations}, 24:1709--1730, (1999).

\bibitem{Ka}
Tosio Kato.
\newblock Remarks on the {E}uler and {N}avier-{S}tokes equations in {${\bf
  R}^2$}.
\newblock In {\em Nonlinear functional analysis and its applications, {P}art 2
  ({B}erkeley, {C}alif., 1983)}, volume~45 of {\em Proc. Sympos. Pure Math.},
  pages 1--7. Amer. Math. Soc., Providence, RI, 1986.

\bibitem{KL}
Tosio Kato and Chi~Yuen Lai.
\newblock Nonlinear evolution equations and the {E}uler flow.
\newblock {\em J. Funct. Anal.}, 56(1):15--28, 1984.

\bibitem{KaPo}
Tosio Kato and Gustavo Ponce.
\newblock Commutator estimates and the {E}uler and {N}avier-{S}tokes equations.
\newblock {\em Comm. Pure Appl. Math.}, 41(7):891--907, 1988.

\bibitem{Maffei2001}
Carlotta Maffei and Carlo Marchioro.
\newblock A confinement result for axisymmetric fluids.
\newblock {\em Rend. Sem. Mat. Univ. Padova}, 105:125--137, 2001.

\bibitem{M94}
Carlo Marchioro.
\newblock Bounds on the growth of the support of a vortex patch.
\newblock {\em Comm. Math. Phys.}, 164:507--524, 1994.

\bibitem{Miller}
Evan Miller.
\newblock On the regularity of axisymmetric, swirl-free solutions of the
  {E}uler equation in four and higher dimensions.
\newblock {\em preprint, arXiv:2204.13406}.

\bibitem{Serfati_pre}
P.~Serfati.
\newblock Borne en temps des caract\'eristiques de l'\'equation d'euler 2d \`a
  tourbillon positif et localisation pour le mod\`ele point-vortex.
\newblock {\em preprint}.

\bibitem{UY1968}
M.~R. Ukhovskii and V.~I. Yudovich.
\newblock Axially symmetric flows of ideal and viscous fluids filling the whole
  space.
\newblock {\em J. Appl. Math. Mech.}, 32:52--61, 1968.

\end{thebibliography}

\end{document}